\documentclass{IEEEtran}
\usepackage{amsfonts}
\usepackage{algorithmic}
\usepackage{textcomp}
\def\BibTeX{{\rm B\kern-.05em{\sc i\kern-.025em b}\kern-.08em
    T\kern-.1667em\lower.7ex\hbox{E}\kern-.125emX}}

\usepackage[noadjust]{cite} 		  
\usepackage{xcolor}       
\usepackage{graphicx}
\DeclareGraphicsExtensions{.eps,.png,.jpg,.jpeg,.bmp,.gif,.pdf}

\usepackage{tabularx,booktabs}	
\usepackage{multirow}	

\usepackage{amsmath}      		
\usepackage{amssymb}	  		
\usepackage{bm}		      		
\usepackage{physics}	  		
\usepackage{soul}				

\newcommand{\R}{\mathbb{R}}		
\newcommand{\C}{\mathbb{C}}		
\newcommand{\transp}{\mathsf{T}}					

\usepackage{amsthm}

\newtheorem{thm}{Theorem}
\newtheorem{cor}{Corollary}

\theoremstyle{definition}

\newtheorem{rem}{Remark}

\newcommand{\arthur}[1]{\textcolor{black}{#1}}
\newcommand{\rev}[1]{#1}

\usepackage{mathtools}

\usepackage[colorlinks,allcolors=blue]{hyperref}

\begin{document}

\title{Duality between controllability and observability for target control and estimation in networks}

\author{Arthur N. Montanari, Chao Duan, \textit{Member, IEEE}, and Adilson E. Motter, \textit{Senior Member, IEEE}
\thanks{The authors acknowledge the US Army Research Office Grant W911NF-19-1-0383 and the use of Quest High-Performance Computing Cluster at Northwestern University. C. D. acknowledges the State Key Development Program for Basic Research of China Grant 2022YFA1004600 and the National Natural Science Foundation of China Grant GQQNKP001. A. N. M. and C. D. contributed equally.}
\thanks{A. N. Montanari is with the Department of Physics and Astronomy and the Center for Network Dynamics, Northwestern University, Evanston, IL 60208, USA (e-mail: arthur.montanari@northwestern.edu).
C.~Duan is with the School of Electrical Engineering, Xi'an Jiaotong University, Xi'an 710049, China (e-mail: cduan@xjtu.edu.cn).
A. E. Motter is with the Department of Physics and Astronomy, the Center for Network Dynamics, the Department of Engineering Sciences and Applied Mathematics, and the Northwestern Institute on Complex Systems, Northwestern University, Evanston, IL 60208, USA (e-mail: motter@northwestern.edu).}
\vspace{-0.3cm}
}

\maketitle


\begin{abstract}       
Output controllability and functional observability are properties that enable, respectively, the control and estimation of \textit{part} of the state vector. These notions are of utmost importance in applications to high-dimensional systems, such as large-scale networks, in which only a target subset of variables (nodes) is sought to be controlled or estimated. Although the duality between full-state controllability and observability is well established, the characterization of the duality between their generalized counterparts remains an outstanding problem. Here, we establish both the weak and the strong duality between output controllability and functional observability. Specifically, we show that functional observability of a system implies output controllability of a dual system (weak duality), and that under a certain geometric condition the converse holds (strong duality). As an application of the strong duality, we derive a necessary and sufficient condition for target control via static feedback. This allow us to establish a separation principle between the design of target controllers and the design of functional observers in closed-loop systems. These results generalize the classical duality and separation principles in modern control theory.
\end{abstract}


\vspace{-0.1cm}


\smallskip
Published in \textit{IEEE Transactions on Automatic Control}, DOI: \href{https://doi.org/10.1109/TAC.2025.3552001}{10.1109/TAC.2025.3552001}.

\vspace{-0.3cm}

\section{Introduction}

\IEEEPARstart{D}{uality} is a mathematical concept that enables the solution of one problem to be mapped to the solution of another problem. The most fundamental example of duality in modern control theory is that between controllability and observability, as introduced by Kalman \cite{Kalman1959}: A system is observable if and only if the dual (transposed) system is controllable. Moreover, when a system is both controllable and observable, the \textit{separation principle} further establishes that the design of a feedback controller and of a state observer are mutually independent. The latter facilitates the design of a closed-loop system with feedback from estimated states.


The analysis of controllability and observability has laid a theoretical foundation for \textit{full-state} controller and observer design, respectively. However, the control and estimation of the entire system state is often unfeasible or not required in high-dimensional systems of current interest, such as large-scale networks \cite{Motter2015,Liu2016,Montanari2020}. 
The unfeasibility may arise from physical and/or costs constraints in the placement of actuators and sensors \cite{Pasqualetti2013,Summers2016} or from a prohibitively high energy required for the operation of these components \cite{Duan2019}. 
These practical limitations led to the development of methods to control and estimate only \textit{part} of the state vector of the system \cite{Darouach2000,Vorotnikov2005}, which subsequently motivated the generalized concepts of output controllability \cite{Bertram1960} and functional observability~\cite{Fernando2010}. These properties characterize the \textit{minimal} conditions that enable the control and estimation of pre-specified lower-order functions of the state variables. 
As a result, these conditions can lead to a substantial order reduction in controller and observer synthesis, as demonstrated in applications to the target control \cite{Gao2014,Casadei2020} and target estimation \cite{Montanari2022,Zhang2023b} of selected subsets of state variables (nodes) in large-scale dynamical networks. 

Unlike the well-established concepts of full-state controllability and observability \cite{Kalman1959}, the notions of output controllability and functional observability emerged in the literature in different contexts, half a century apart \cite{Bertram1960,Fernando2010}. Yet, despite their shared goals (to control and estimate a lower-order function, respectively), the existence of a mapping between the functional observability of a system and the output controllability of another (dual) system is yet to be established. 
This is the case because the relation between these generalized properties, as well as between the set of observable and controllable state variables \cite{Iudice2019}, does not follow straightforwardly from the classical duality principle. 
Consequently, algorithms developed for optimal actuator placement in output controllability \cite{Gao2014,Waarde2017,Czeizler2018,commault2019functional} cannot be employed with guaranteed performance for optimal sensor placement in functional observability \cite{Montanari2022}, and vice versa. 
This is in contrast with the case of full-state controllability and observability, in which their duality enables a single algorithm to solve both placement problems \cite{Liu2011,Siami2020}.

In this paper, we establish a duality principle  between output controllability and functional observability. Motivated by the notions of weak and strong duality in optimization theory\footnote{In optimization theory, the ``weak'' duality principle establishes that the optimal solution of a dual problem provides a bound to the solution of a primal problem. The notion of ``strong'' duality usually requires additional conditions and concerns scenarios where the difference between the solutions of the primal and the dual problem (also known as the duality gap) is zero.} \cite{balinski1969duality}, 
we derive a weak duality between both properties in which the functional observability of a system implies the output controllability of a dual (transposed) system. Moreover, under a specific condition, we establish the strong duality in which the converse also holds. 
For output controllable or functionally observable systems, we propose Gramian-based measures to quantify the amount of energy required to control or estimate a lower-order function.
Strong duality leads to the necessary and sufficient conditions for target control via static feedback: When a system is output controllable and its dual system is functionally observable, we show that static feedback control of a linear function of the state variables is possible. Based on these results, we then establish a separation principle in closed-loop systems between the design of a feedback target controller and that of a functional observer.
These contributions are shown to be the natural extensions of the duality and separation principles in modern control theory.


\vspace{-0.1cm}
\section{Preliminaries}
\label{sec.prelim}


Consider the linear time-invariant dynamical system
\begin{align}
    \dot{\bm x} &= A\bm x + B\bm u, 
    \label{eq.dynsys}
    \\
    \bm y &= C\bm x,
    \label{eq.output}
\end{align}
where $\bm x\in\R^n$ is the state vector, $\bm u\in\R^p$ is the input vector, $\bm y\in\R^q$ is the output vector, $A\in\R^{n\times n}$ is the system matrix, $B\in\R^{n\times p}$ is the input matrix, and $C\in\R^{q\times n}$ is the output matrix.
%
%
The \textit{linear function} of the state variables
\begin{equation}
    \bm z = F\bm x
    \label{eq.target}
\end{equation}
defines the \textit{target} vector $\bm z\in\R^r$ sought to be controlled or estimated, where $F\in\R^{r\times n}$ is the functional matrix ($r\leq n$).

Let the \emph{\rev{output controllable set}} $\mathcal{C}_F$ of a system \eqref{eq.dynsys}--\eqref{eq.target} be the set of all reachable states $\bm{z^*} \in \mathbb{R}^r$ for which there exists an input $\bm{u}(t)$ that steers the system, in finite time $t\in[0,t_1]$, from an initial state $\bm{x}(0)={0}$ to some final state $\bm{x}(t_1)$ satisfying $\bm{z}(t_1) = F\bm{x}(t_1) := \bm z^*$. The system represented by the triple $(A,B;F)$ is output controllable when $\mathcal{C}_F=\R^r$. A necessary and sufficient condition for output controllability is \cite{Bertram1960,Lazar2020}
\begin{equation}
    \operatorname{rank}(F\mathcal C) = \operatorname{rank}(F),
    \label{eq.outputctrb}
\end{equation}
where $\mathcal C = [B \,\, AB \,\, A^2B \,\, \ldots \,\, A^{n-1}B]$ is the controllability matrix. In particular, condition \eqref{eq.outputctrb} is necessary and sufficient for the invertibility of $F W_c(t_1)F^\transp$, 
which guarantees that, for each $\bm{z^*}\in \mathcal{C}_F$, there exists a control law
\begin{equation}
    \bm u(t) = B^\transp e^{A^\transp(t_1-t)} F^\transp (F W_c(t_1)F^\transp)^{-1} \bm z^*
\label{eq.statedependentcontrol}
\end{equation}
capable of steering the target vector from $\bm z(0)={0}$ to $\bm z(t_1)=\bm{z^*}$, where $W_c(t)=\int_0^{t} e^{A(t-\tau)}BB^\transp e^{A^\transp(t-\tau)}\text{d}\tau$ is the controllability Gramian. \arthur{The set of reachable states $\bm z^*$ is thus given by $\mathcal{C}_F = F\operatorname{Im}(W_c)$.}

\begin{rem}     \label{remark.output-vs-target}
    The nomenclature ``output'' controllability was originally motivated by applications in which the target vector sought to be controlled was precisely the output ($F=C$) \cite{Bertram1960,Lazar2020}. 
    However, this property can be generally defined for any $F$. This led to the notion of partial stability and control \cite{Vorotnikov2005} as well as the contemporaneous terminology target controllability for networks where only a subset of variables (nodes) is sought to be controlled \cite{Gao2014,Waarde2017,Czeizler2018,commault2019functional,li2020structural,Shali2022}.
\end{rem}


In addition, let the \emph{\rev{functionally observable set}} $\mathcal{O}_F$ of system \eqref{eq.dynsys}--\eqref{eq.target} be the set of all $\bm z^*\in\R^r$ such that the state $\bm z(0) = F\bm x(0) := \bm z^*$ can be uniquely determined from the output $\bm y(t)$ and input $\bm u(t)$ signals over $t\in[0,t_1]$. The system represented by the triple $(C,A;F)$ is functionally observable when $\mathcal{O}_F=\mathbb{R}^r$. A necessary and sufficient condition for this property is

\begin{equation}
    \operatorname{rank}\left(
    \begin{bmatrix}
        \mathcal O \\ F
    \end{bmatrix}
    \right)
    =
    \operatorname{rank}(\mathcal O),
    \label{eq.functobsv}
\end{equation}
where $\mathcal O = [C^\transp \, (CA)^\transp \, (CA^2)^\transp \, \ldots \, (CA^{n-1})^\transp]^\transp$ is the~observability matrix (see proof in \cite[Th.~5]{Jennings2011}, \cite[Sec.~I]{Rotella2016}). Given the observability Gramian $W_o(t)=\int_0^t e^{A^\transp\tau}C^\transp C e^{A\tau}{\rm d}\tau$, condition \eqref{eq.functobsv} is also necessary and sufficient for the existence of some matrix $G\in\R^{r\times n}$ such that $GW_o(t_1)=F$. This guarantees reconstruction of $\bm z(0)$ in finite time \cite{Montanari2022}, that is,
\begin{align}
         W_o(t_1)\bm x(0) &= \int_0^{t_1} e^{A^\transp t}C^\transp \bm{h}(t,\bm{u},\bm{y})  \text{d}t
         \label{eq.reconstructible_x0}
\end{align}
implies $\bm z(0) = G\int_0^{t_1} e^{A^\transp t}C^\transp \bm{h}(t,\bm{u},\bm{y}) \text{d}t$, 
where $\bm{h}(t,\bm{u},\bm{y}) =\bm{y}(t)- C\int_0^t e^{-A \tau} \bm{B} \bm{u}(\tau)  {\rm d}\tau $ is a functional of the input $\bm{u}(t)$ and output $\bm{y}(t)$ over $t\in [0,t_1]$. \arthur{Thus, $\mathcal O_F = \{\bm z^*\in\R^r : \bm z^* \perp F\operatorname{Ker}(W_o) \,\, \text{and} \,\, \bm z^*\in F\Im(W_o) \}$ (see proof of Theorem~\ref{theor.weakduality}).}


Throughout, we use $(A,B;F)$ to analyze the output controllability of a system with system matrix $A$, input matrix $B$, and functional matrix $F$. The term ``target control'' refers to the control of the target vector $\bm z(t)$ via an input $\bm u(t)$. Likewise, we use $(C,A;F)$ to analyze the functional observability of a system with output matrix $C$, system matrix $A$, and functional matrix $F$, and ``target estimation'' refers to the estimation of $\bm z(t)$ from an output $\bm y(t)$. {Furthermore, we assume that $F$ has linearly independent rows (i.e., $\rank(F)=r$), which guarantees, when output controllability holds, the existence of an input $\bm u(t)$ that can drive independently every component $\bm z_i(t)$ of the target vector to an arbitrary state $\bm z_i^\star$.}

\section{Duality Principle}
\label{sec.duality}

We now generalize the concept of duality between controllability and observability. Consider a pair of systems $(C,A;F)$ and $(A^\transp,C^\transp;F)$. We show that the functional observability of the former is in general a sufficient (but unnecessary) condition for the output controllability of the latter. This result will thus establish the \emph{weak duality} between functional observability and output controllability. By further imposing a certain condition on the system matrices, the functional observability of $(C,A;F)$ becomes equivalent to the output controllability of $(A^\transp,C^\transp;F)$, which we call the \emph{strong duality}.

In what follows, let $\mathcal O_F$ be the functionally observable set of the system $(C,A;F)$ and $\mathcal C_F$ be the output controllable set of the dual system $(A^\transp,C^\transp;F)$. Given any final time $t_1$, the observability Gramian $W_o(t_1)$ of system $(C,A)$ and the controllability Gramian $W_c(t_1)$ of system $(A^\transp,C^\transp)$ coincide, and we denote it by $W = \int_{0}^{t_1}e^{A^\transp\tau}C^\transp C e^{A\tau} d \tau$. Let $W=U\Lambda U^H$ be the eigendecomposition of the symmetric matrix $W$, where $\Lambda$ is the diagonal of eigenvalues and $U$ is the unitary matrix of eigenvectors. We partition $U$ as $U = [U_1 \,\,\, U_2]$, where $U_1$ and $U_2$ consist of columns corresponding to nonzero and zero eigenvalues, respectively. Clearly, $W = U_1 \Lambda_1 U_1^H$, where $\Lambda_1$ is the diagonal matrix of nonzero eigenvalues. We have that $\operatorname{Im}(W) = \operatorname{Im}(U_1)$ and $\operatorname{Ker}(W) = \operatorname{Im}(U_2)$.
\arthur{Note that the image and null spaces of $W$ are independent of the specific choice of basis for $U_1$ and $U_2$.}

\begin{thm}[Weak duality]
    \label{theor.weakduality}
For any given pair of systems $(C,A;F)$ and $(A^\transp,C^\transp;F)$, the relation $\mathcal{O}_F \subseteq \mathcal{C}_F$~holds. Therefore, if $(C,A;F)$ is functionally observable, then $(A^\transp,C^\transp;F)$ is output controllable.
\end{thm}

\begin{proof}
%
For system $(C,A;F)$, Eq. \eqref{eq.reconstructible_x0} is equivalent to
\begin{equation}
    U_1^H\bm{x}(0) = \Lambda_1^{-1}U_1^H\int_0^{t_1} e^{A^{\transp} t}C^{\transp} \bm{h}(t,\bm{u},\bm{y}) \text{d}t .
\end{equation}
Thus, $\bm{z}(0) = F\bm{x}(0)=FU_1U_1^H\bm{x}(0)+FU_2U_2^H\bm{x}(0)$, yielding
\begin{equation}\label{eqn:z0decompose}
\begin{aligned}
     \bm{z}(0)  =& \,\, FU_1 \Lambda_1^{-1}U_1^H\int_0^{t_1} e^{A^{\transp} t}C^{\transp} \bm{h}(t,\bm{u},\bm{y}) \text{d} t \\&  + F U_2 U_2^H \bm{x}(0).
\end{aligned}
\end{equation}
Since $\bm{x}(0)$ is arbitrary and unknown, it follows from Eq.~\eqref{eqn:z0decompose} that $\bm{z}(0)$ can be uniquely determined by the (known) functional $\bm h(t,\bm{u},\bm{y})$ \rev{only if $FU_2U_2^H\bm x(0) = 0$, which is equivalent to the condition $\bm{z}(0) \perp F\operatorname{Im}(U_2)$. Thus, Eq.~\eqref{eqn:z0decompose} implies that the necessary and sufficient condition for $\bm z(0)$ to be uniquely determined from $\bm h(t,\bm u,\bm z)$ is that $\bm{z}(0) \perp F\operatorname{Im}(U_2)$ and $\bm{z}(0) \in F\operatorname{Im}(U_1)$.} Consequently,
\begin{equation}\label{eqn:OfCfrelation}
\begin{aligned}
    \mathcal{O}_F = \operatorname{Im}(FU_1) \cap  (F\operatorname{Im}(U_2))^{\perp} = \mathcal{C}_F \cap (F\operatorname{Ker}(W))^{\perp}.
    \end{aligned}
\end{equation}
%
It follows that $\mathcal{O}_F \subseteq \mathcal{C}_F$. If $(C,A;F)$ is functionally observable, then $\mathcal {O}_F=\R^r$ and hence $\mathcal C_F = \R^r$. Thus, by definition, $(A^\transp,C^\transp;F)$ is output controllable.
\end{proof}

\begin{thm}[Strong duality]
    \label{theor.strongduality}
For any given pair of systems $(C,A;F)$ and $(A^\transp,C^\transp;F)$, the relation $\mathcal{O}_F = \mathcal{C}_F$ holds if and only if $F\operatorname{Im}(W) \perp F\operatorname{Ker}(W)$. Under this condition, $(C,A;F)$ is functionally observable if and only if $(A^\transp,C^\transp;F)$ is output controllable.
\end{thm}

\begin{proof}
According to Eq.~\eqref{eqn:OfCfrelation}, $\mathcal{O}_F = \mathcal{C}_F$ if and only if $\mathcal{C}_F \subseteq (F\operatorname{Ker}(W))^{\perp}$, which is equivalent to $F\operatorname{Im}(W) \perp F\operatorname{Ker}(W)$. Moreover, if $\mathcal C_F = \mathcal O_F = \R^r$, then by definition $(C,A;F)$ is functionally observable and $(A^\transp,C^\transp;F)$ is output controllable.
\end{proof}


\begin{rem} \label{rem.classicalduality}
    Theorem~\ref{theor.strongduality} generalizes the classical duality principle between full-state controllability and observability \cite[Th. 6.5]{Chi-TsongChen1999}. To see this, note that, for $F=I_n$ (where $I_n$ is the identity matrix), conditions \eqref{eq.outputctrb} and \eqref{eq.functobsv} reduce respectively to the classical conditions of full-state controllability ($\operatorname{rank}(\mathcal C) = n$) and full-state observability ($\operatorname{rank}(\mathcal O) = n$). Furthermore, the condition in Theorem~\ref{theor.strongduality} reduces to $\operatorname{Im}(W)\perp\operatorname{Ker}(W)$, which always holds since $W$ is symmetric. Therefore, $\mathcal O_{I_n}=\mathcal C_{I_n}$.
\end{rem}

Theorems \ref{theor.weakduality} and \ref{theor.strongduality} provide a geometric characterization of the duality in terms of the inclusion relation between the sets $\mathcal{C}_F$ and $\mathcal{O}_F$.  
These results have a direct interpretation in the geometric approach to control theory \cite{wonham1979book,basile1991book}. Condition~\eqref{eq.outputctrb} implies that $(A^\transp,C^\transp;F)$ is output controllable if and only if $F\langle A^\transp | \Im(C^\transp)\rangle = \R^r$ \cite{basile1969controlled}, where the controllable subspace $\langle A^\transp | \Im(C^\transp)\rangle = \sum_{i=1}^n (A^\transp)^{i-1}\Im(C^\transp)$ is the smallest $A^\transp$-invariant subspace containing $\Im(C^\transp)$.
Likewise, following condition \eqref{eq.functobsv}, $(C,A;F)$ is functionally observable if and only if $F\mathcal N = \{0\}$, where the unobservable subspace $\mathcal N = \bigcap_{i=1}^{n}\operatorname{Ker}(CA^{i-1})$ is the largest $A$-invariant subspace contained in $\operatorname{Ker}(C)$. By duality, $\mathcal N = \langle A^\transp | \operatorname{Im}(C^\transp) \rangle^\perp$.
From the proof of Theorem \ref{theor.weakduality}, we have that the set $\mathcal C_F = F\langle A | \Im(B)\rangle$ is a subspace. In turn, the set $\mathcal O_F$ is determined by Eq.~\eqref{eqn:OfCfrelation}, showing that $\mathcal O_F$ is related to the \textit{orthogonal complement} of $F\mathcal N$ and hence is also a subspace.
%
%
Specifically, strong duality holds if and only if $\gamma = 0$, where we define $\gamma = \operatorname{dim} (\mathcal{C}_F) - \operatorname{dim} (\mathcal{O}_F) \geq 0$ as the duality gap.

\begin{figure}[t!]
\begin{center}
\includegraphics[width=\columnwidth]{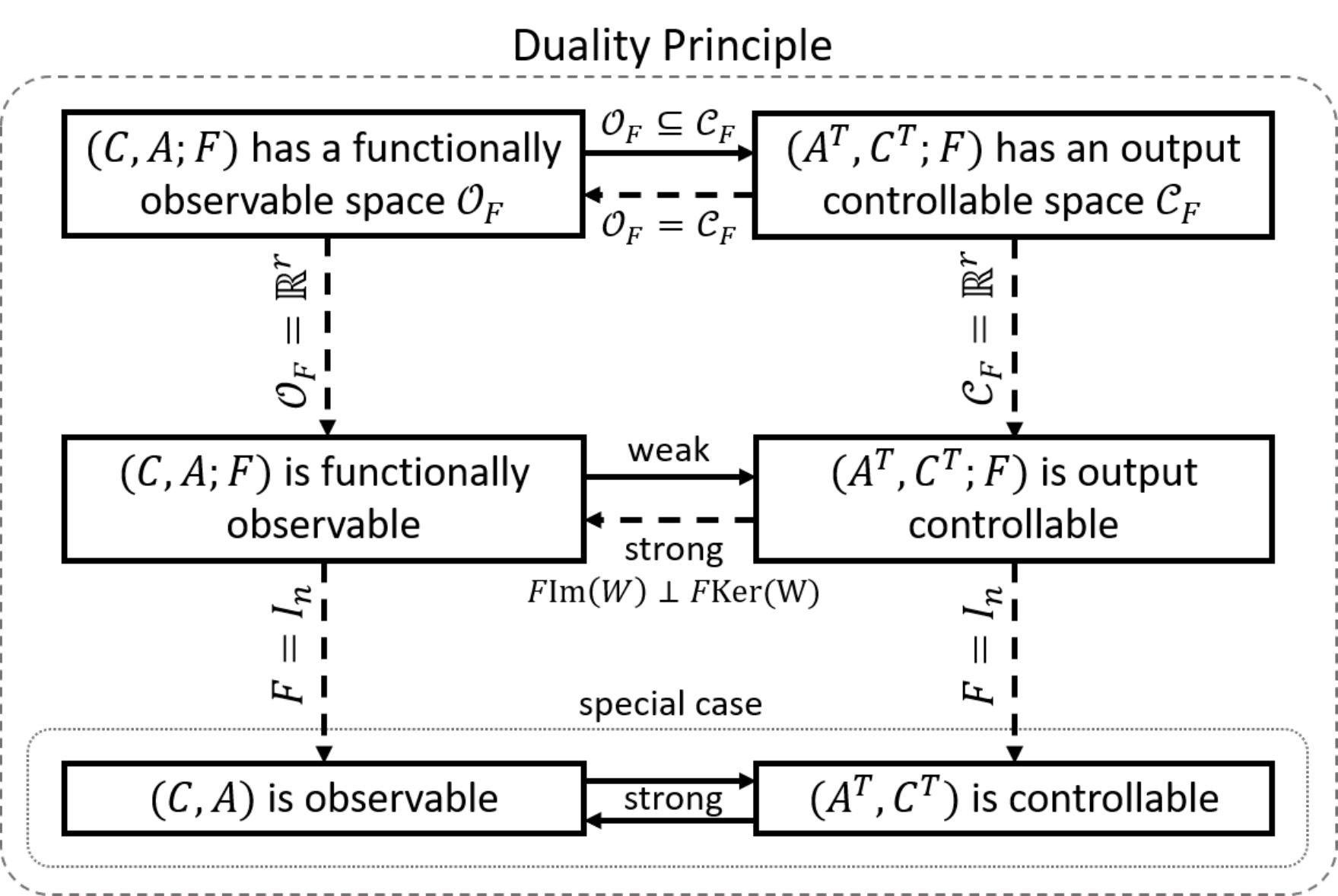} 
\caption{\label{fig.diagram} \rev{Diagram of the duality principle between functional observability and output controllability. Solid arrows indicate one property implies another, while dashed arrows indicate conditional implications.}}
\vspace{-0.6cm}
\end{center}                        
\end{figure}

%

Fig.~\ref{fig.diagram} shows a diagram of the established duality principle. We note that strong duality may hold even if the two systems are neither output controllable nor functionally observable; for example, the trivial case where $C = 0$ and $W = 0$, and hence $F\operatorname{Im}(0)\perp F\operatorname{Ker}(0)$. 
The following corollaries provide sufficient conditions for the strong duality principle.

\begin{cor} 
A pair of systems $(C,A;F)$ and $(A^\transp,C^\transp;F)$ is strongly dual (i.e., $\mathcal{O}_F = \mathcal{C}_F$) if $(C,A;F)$ is functionally observable.
\label{cor.functobsvstrongdual}
\end{cor}

\begin{proof} Since $(C,A;F)$ is functionally observable, condition \eqref{eq.functobsv} is satisfied. Hence, $\operatorname{row}(F)\subseteq \operatorname{row}(\mathcal O)$ and $F\operatorname{Ker}(\mathcal O) = \{ 0 \}$, where $\operatorname{row}(\cdot)$ denotes the row space of a matrix. Given that $\operatorname{Im}(W) = \operatorname{Im}(\mathcal O)$, it follows that the strong duality condition in Theorem~\ref{theor.strongduality} holds trivially. \end{proof}

\begin{cor} \label{cor.conformal}
\rev{A pair of systems $(C,A;F)$ and $(A^\transp,C^\transp;F)$ is strongly dual if $\bm z = F\bm x$ is a conformal linear transformation.}
\end{cor}

\begin{proof}
\rev{
    By definition, $F$ represents a conformal linear transformation if and only if the transformation preserves the angles between any two vectors, i.e.,
$
        \frac{\langle \bm x, \tilde{\bm x} \rangle}{|\bm x||\tilde{\bm x}|}
        =
        \frac{\langle F\bm x, F\tilde{\bm x} \rangle}{|F\bm x||F\tilde{\bm x}|}
$
    holds for all $\bm x,\tilde{\bm x}\in\R^n$. Since $\Im(W)\perp \operatorname{Ker}(W)$, the orthogonality $F\Im(W)\perp F\operatorname{Ker}(W)$ is preserved.
}
\end{proof}



\subsection{Example}
\label{sec.dualityex}

Consider the 5-dimensional dynamical system given by
\begin{equation}
    A =
    \begin{bmatrix}
    -1 & 0 & 0 & 0 & -1 \\
    -1 & a_{22} & 0 & 0 & 0 \\
    0 & 0 & a_{33} & 0 & 0 \\
    0 & 0 & 0 & a_{44} & 0 \\
    0 & -1 & -1 & -1 & -1
    \end{bmatrix},
    \,\,\,
    \begin{matrix}
    C =
    \begin{bmatrix}
    0 & 0 & 0 & 0 & 1    
    \end{bmatrix}, \\ \\
    F =
    \begin{bmatrix}
    0 & 1 & 0 & 0 & 0    
    \end{bmatrix}. \\
    \end{matrix}
    \label{eq.pairACexample}
\end{equation}
%
%
Let $a_{22}\neq 0$ and $a_{33}=a_{44}=0$ except when stated otherwise. Fig.~\ref{fig.graph} provides a graph representation of the system $(C,A;F)$ and its dual $(A^\transp,C^\transp;F)$. The observability matrix is given by

\begin{equation*}
    \mathcal O =
    \begin{footnotesize}
    \left[\begin{array}{*{5}{c@{\hspace{4pt}}}}
        0 & 0 & 0 & 0 & 1 \\
        0 & -1 & -1 & -1 & -1 \\
        1 & -a_{22}+1 & 1 & 1 & 1 \\
        a_{22}-2 & -a_{22}^2+a_{22}-1 & -1 & -1 & -2 \\
        a_{22}^2-2a_{22}+3 & -a_{22}^3 + a_{22}^2 - a_{22} + 2 & 2 & 2 & - a_{22}+4
    \end{array}\right]
    \end{footnotesize}
    .
\end{equation*}
Note that $(C,A;F)$ is not full-state observable ($\operatorname{rank}(\mathcal O)<5$) but is functionally observable given that $\operatorname{row}(F)\subseteq\operatorname{row}(\mathcal O)$.
%
%
%
Therefore, following Theorem \ref{theor.weakduality}, $(A^\transp,C^\transp;F)$ is output controllable. 
The converse relation, however, is not always true: output controllability of a system is not sufficient for the functional observability of the dual system. By considering $a_{22} = 0$ (equivalent to the absence of a self-edge in node 2 in Fig.~\ref{fig.graph}), it can be verified that $(A^\transp,C^\transp;F)$ remains output controllable, whereas $(C,A;F)$ loses functional observability. This is a consequence of the controllable and observable spaces defined by the Gramian.
Recall that $\operatorname{Im}(W) = \operatorname{Im}(\mathcal O) = \operatorname{Im}(\mathcal C^\transp)$. For $a_{22} = 0$, $F\operatorname{Im}(W) = F\operatorname{Ker}(W) = \operatorname{span}\{1\}$, where $\operatorname{span}\{\cdot\}$ is the span of a set of vectors.
From Theorem \ref{theor.strongduality}, strong duality does not hold and $(C,A;F)$ is not functionally observable. Indeed, strong duality is possible only if $a_{22}\neq 0$, where $F\operatorname{Ker}(W)= \{ 0 \}$ and thus $F\operatorname{Im}(W)\perp F\operatorname{Ker}(W)$. 

\begin{figure}
    \centering
    \includegraphics[width=0.83\columnwidth]{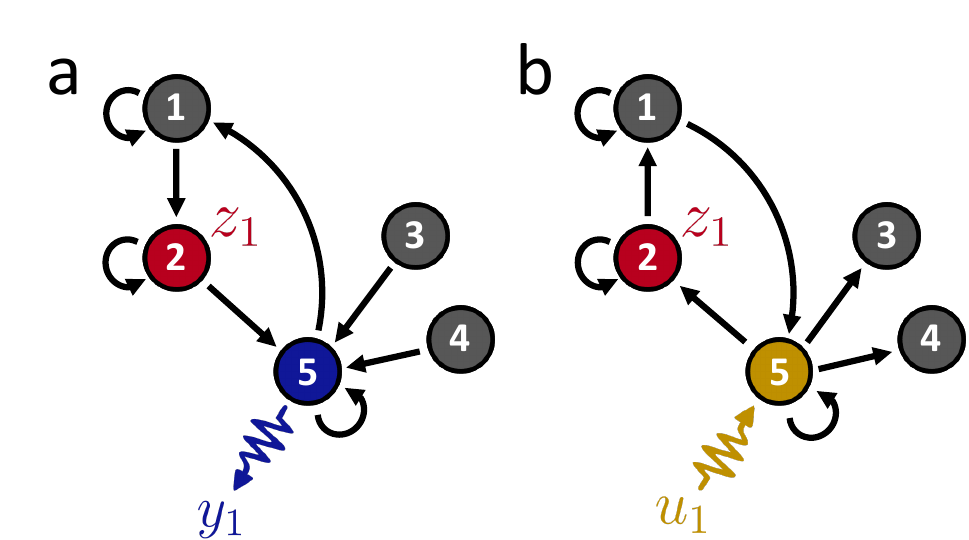}
    \caption{
    \label{fig.graph} Graph representation of a pair of dual systems.
    (a)~System $(C,A;F)$. 
    (b)~Dual system $(A^\transp,C^\transp;F)$.
    In each graph, nodes $i=1,\ldots,n$ represent the state variables $x_i$ and an edge is drawn from node $j$ to node $i$ if $a_{ij}\neq 0$. Note that transposing $A$ in the dual system is equivalent to reversing the direction of the edges. 
    The input, output, and target variables are colored in yellow, blue, and red, respectively. }
\vspace{-0.5cm}
\end{figure}

\subsection{Target control and observation energy in networks}
\label{sec.examp.network}

\rev{
In applications to large-scale networks, testing the output controllability and functional observability of a system using the algebraic rank conditions \eqref{eq.outputctrb} and \eqref{eq.functobsv}, respectively, can be prone to numerical challenges due to the poor conditioning of the controllability and observability matrices for large $n$. To circumvent these challenges, generic notions of output controllability \cite{Waarde2017,Czeizler2018,commault2019functional,li2020structural,Shali2022} and functional observability \cite{Montanari2022,zhang2023functional} for \textit{structured systems} have been proposed in the literature.
Such studies established intuitive graph-theoretic conditions for these properties based on the network structure, the set of actuators (or sensors), and the set of targets\textemdash respectively encoded by matrices $A$, $B$ (or $C$), and $F$. This approach has led to computationally efficient algorithms for the minimum placement of actuators and sensors in large-scale networks \cite{Montanari2023target}, and is particularly suitable for systems involving modeling uncertainties in the parameters (when the structure of $A$ is reliably known but the exact numerical entries are not).}
%
%
However, in addition to verifying whether a system is output controllable or functionally observable, it is important to address the question of how ``hard'' it is to control or observe the target vector $\bm z(t)$. For practical purposes, and in accordance with the existing literature \cite{Pasqualetti2013,Summers2016,Casadei2020}, we measure this hardness in terms of energy. We propose the following measures of the control (observation) energy required to drive (observe) a target vector. 

\subsubsection{Target control energy}

Consider an output controllable system $(A,B;F)$. Let $\bm u(t)$ be determined by Eq.~\eqref{eq.statedependentcontrol}, which is the control signal that requires the smallest amount of energy $\int_0^{t_1} \norm{\bm u(t)}^2 {\rm d}t$ to drive a target vector from the initial state $\bm z(0) := F\bm x(0) = 0$ to a given final state $\bm z^* := F\bm x^*$, where $\bm x^*:=\bm x(t_1)$. We refer to $\bm z^*$ for which the minimum energy is the largest as the \textit{worst-case scenario} and the corresponding energy as the maximum \textit{target control energy} \cite{Casadei2020}:
%
\begin{equation} \label{eq.controlenergy}
        E_{\rm tc} :=
        \max_{\norm{F\bm x^*} = 1}   \int_0^{t_1} \norm{\bm u(t)}^2 {\rm d}t = \frac{1}{\lambda_{\rm min}(FW_cF^\transp)}, 
\end{equation}
where $\lambda_{\rm min}(\cdot)$ denotes the minimum eigenvalue of a matrix. This expression can be derived by considering Eq. \eqref{eq.statedependentcontrol} for the initial state $\bm x(0) = 0$, which leads to $\int_0^{t_1} \norm{\bm u(t)}^2 {\rm d}t = (\bm z^*)^\transp (FW_cF^\transp)^{-1}\bm z^* \leq \lambda_{\rm max}\left((FW_cF^\transp)^{-1}\right) \norm{\bm z^*}$.
%
%
%
By imposing the constraint $\norm{F\bm x^*}=1$, Eq. \eqref{eq.controlenergy} corresponds to the direction in the row space of $F$ that is the hardest to control. 
If the system is not output controllable, then $FWF^\transp$ is singular and hence $\int_0^{t_1} \norm{\bm u(t)}^2 {\rm d}t$ is undefined. 

\begin{figure*}
    \centering
    \includegraphics[width=0.85\linewidth]{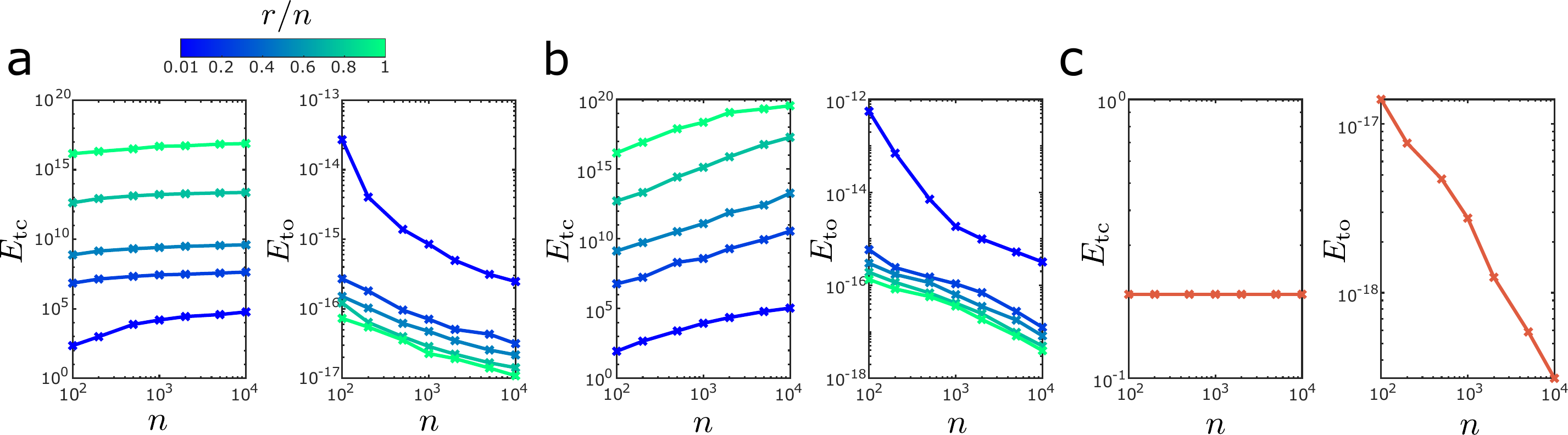}
    \caption{\rev{Target control energy $E_{\rm tc}$ and target observation energy $E_{\rm to}$ as a function of the network size $n$ in (a) scale-free networks, (b) small-world networks, and (c) chain networks. Each data point represents the average energy across 100 independent realizations of the network models.
    For each realization, the state variables $\bm x_i$ representing nodes in the network are actuated independently in the left panel (i.e., $B\in\R^{n\times p}$ has only one nonzero entry per column) and measured independently in the right panel (i.e., $C\in\R^{q\times n}$ has only one nonzero entry per row). 
    In (a) and (b), the target variables are also selected independently (i.e., $F\in\R^{r\times n}$ has only one nonzero entry per row), and the nonzero entries of $F$, $B$, and $C$ are selected randomly. The curves are color coded according to the fraction of target nodes $r/n$, for $p=q=0.1n$. 
    In (c), we set $p=q=r=1$, where only the first node is actuated in the left panel (i.e., $B = [1 \,\, 0 \,\, 0 \,\, \ldots \,\, 0]^\transp$) and measured in right panel (i.e., $C = B^\transp$). The target variable is defined as $z = \sum_{i=1}^n x_i$ (i.e., $F = [1 \,\, 1 \,\, 1 \,\, \ldots \,\, 1]$).
    In all simulations, the Gramian $W$ is computed using the system matrix $A = -(L + \alpha I_n)$, where $L$ represents the Laplacian matrix of the (symmetric) adjacency matrix $G$, $\alpha=0.1$ is set to ensure numerical stability of the matrices for large systems, and $G_{ij}\sim\mathcal U[0,1]$ if node $j$ is connected to $i$ and $G_{ij}=0$ otherwise. The scale-free networks were generated using the Barabási–Albert model, where each new node is preferentially attached to 3 existing nodes, and the small-world networks were generated using the Newman–Watts model, where the probability of adding a new edge to a ring graph is 0.2 \cite{Newman2010}.
    }} 
    \vspace{-0.55cm}
    \label{fig.complexnets}
\end{figure*}

\subsubsection{Target observation energy}

Consider a functionally observable system $(C,A;F)$. We propose a measure to quantify the energy contribution of the initial target state $\bm{z}^* := F\bm{x}^*$ to the output signal, where $\bm{x}^* := \bm{x}(0)$ represents the initial state to be reconstructed. For the worst-case scenario, we call this measure the minimum \textit{target observation energy}:
\begin{equation} \label{eq.obsvenergy}
        E_{\rm to} :=
        \min_{\norm{F\bm x^*} = 1}  \int_0^{t_1} \norm{\bm y(t)}^2 {\rm d}t = \frac{1}{\lambda_{\rm max}(FW_o^\dagger F^\transp)},
\end{equation}
\noindent
where we assume that $A$ is Hurwitz stable so that the output signal is bounded.
To solve the optimization \eqref{eq.obsvenergy}, consider the Lagrangian $\mathcal L(\bm x^*,\lambda) =  \int_0^{t_1} \norm{\bm y(t)}^2{\rm d}t + \lambda(\norm{F\bm x^*}^2 - 1)$, where $\lambda$ is the Lagrangian multiplier and $\int_0^{t_1} \norm{\bm y(t)}^2{\rm d}t = (\bm x^*)^\transp W_o \bm x^*$. The critical points $({\bm x}^*, {\lambda})$ of $\mathcal L$ are given by
\begin{equation} \label{eq.lagrangiangrad}
        \gradient_{\bm x^*} \mathcal L(\bm x^*,\lambda) = (W_o - \lambda F^\transp F){\bm x}^* = 0.
\end{equation}
\noindent
Evaluating the cost function at $({\bm x}^*, {\lambda})$ yields $\int_0^{t_1} \norm{\bm y(t)}^2{\rm d}t = ({\bm x}^*)^\transp W_o {\bm x}^* = ({\bm x}^*)^\transp \lambda F^\transp F {\bm x}^* = \lambda \norm{F{\bm x}^*}^2 = \lambda$, where we imposed the constraint $\norm{F\bm x^*}=1$. 
From Eq.~\eqref{eq.lagrangiangrad}, it follows that the Lagrangian multiplier $\lambda$ is an eigenvalue of the matrix pencil $W_o-\lambda F^\transp F$, also denoted $(W_o,F^\transp F)$. Thus, $E_{\rm to} = \lambda_{\rm min}(W_o,F^\transp F)$, which is the smallest eigenvalue of the matrix pencil $(W_o,F^\transp F)$. Multiplying Eq. \eqref{eq.lagrangiangrad} on the left by $G$ and recalling that $GW_o = F$, we obtain $\left(GF^\transp - \frac{1}{\lambda} I_r \right)\bm z_0 = 0$.
Therefore, $E_{\rm to} = \lambda_{\rm min}(W_o,F^\transp F) = \left[\lambda_{\rm max}(GF^\transp)\right]^{-1}$, where $G = FW_o^\dagger$ and $W_o^\dagger$ denotes the right pseudoinverse of $W_o$ (i.e., $W_oW_o^\dagger = I_n$).
Note that if the system is not functionally observable, there does not exist a matrix $G$ such that $GW_o = F$ and hence $E_{\rm to}$ is undefined.

Small $E_{\rm to}$ implies a small contribution of $\bm z^*$ to the output energy $\int_0^{t_1} \norm{\bm y(t)}^2 {\rm d}t$, which may be obscured by practical limitations (e.g., numerical errors) and lead to poor target estimation. To see this, consider the sequence of measurements $\bm y(k)$, $k \in \{0,\Delta t, 2\Delta t, \ldots, T\}$, where $\Delta t$ is the sampling time and $T$ is the final time. Let $[\bm y(0)^\transp \,\, \bm y(\Delta t)^\transp \,\, \ldots \,\, \bm y(T)^\transp ]^\transp  = [C^\transp  \,\, (Ce^{A\Delta t})^\transp  \,\, \ldots \,\, (Ce^{AT})^\transp ]^\transp \bm x(0) =: \Psi \bm x(0)$. Since the system is functionally observable,  $\bm z(0)$ can be estimated from the least-square problem $[\bm y(0)^\transp  \ldots  \bm y(T)^\transp ]^\transp = \Phi\bm z(0)$, where 
$\Phi = F\Psi^\dagger$. The estimation accuracy of $\bm z(0)$ thus depends on the condition number of $\Phi$, determined by $\kappa := \sqrt{\lambda_{\rm max}(\Phi\Phi^\transp)/\lambda_{\rm min} (\Phi\Phi^\transp)}.$
It follows that $\Phi\Phi^\transp = F\Phi^\dagger(\Phi^\dagger)^\transp F^\transp =  F(\Phi\Phi^\transp)^\dagger F = F(\sum_{k}e^{A^\transp k\Delta t}C^\transp C e^{Ak\Delta t})^\dagger F^\transp \approx FW_o^\dagger F^\transp$ for a sufficiently small $\Delta t$. This reveals a relation between the energy $E_{\rm to}$, 
the condition number $\kappa$, and the estimation accuracy of $\bm z(0)$.


\vspace{-0.1cm}
\begin{rem} \label{rem.energyduality}
    \rev{The  measures $E_{\rm tc}$ and $E_{\rm to}$ are not directly related in general, despite the equivalence between the controllability and observability Gramians ($W_c=W_o=:W$). As a special case, when $F$ has orthonormal rows (i.e., $FF^\transp = I_r$), we have} $F^\transp = F^\dagger$ and $FW^\dagger F^\transp = (F^\transp)^\dagger W^\dagger F^\dagger = (FWF^\transp)^\dagger$. 
    By assumption, $E_{\rm tc}$ and $E_{\rm to}$ are defined respectively for output controllable and functionally observable systems. Thus, $FWF^\transp$ is nonsingular and $(FWF^\transp)^\dagger = (FWF^\transp)^{-1}$, yielding $E_{\rm to} = E_{\rm tc}^{-1}$.
    This case suggests that the degree of difficulty to observe a system goes hand in hand with the difficulty to control the dual system. However, such a simple relation is not expected for a general pair of dual systems, as illustrated in the example below. For the previously considered case of $F=I_n$ \cite{Pasqualetti2013,Summers2016}, we obtain the energy measures $E_{\rm tc} = \left[\lambda_{\rm min}(W_c)\right]^{-1}$ for full-state controllability and $E_{\rm to} = \left[\lambda_{\rm max}(W_o^{-1})\right]^{-1} = \lambda_{\rm min}(W_o)$ for full-state observability. 
    
\end{rem}

\subsubsection{Example}

Fig. \ref{fig.complexnets}a,b presents the scaling of the  energy measures for scale-free and small-world networks of increasing size. The target control (observation) energy increases (decreases) on average with the network size $n$. As expected, the increase rate depends on the network structure and the proportion of target variables. 
%
Nonetheless, the energy required to drive a \textit{target} vector is much smaller than the energy required to drive the \textit{full-state} vector when $r\ll n$. 
Likewise, the output energy retains much more information of the target state than of the full state, as shown by larger  $E_{\rm to}$ for $r\ll n$. 
%
%

The results in Fig. \ref{fig.complexnets}a,b are generated for a random selection of actuated, measured, and targeted variables in the network. In this case, the anticorrelated trend in the dependence of $E_{\rm tc}$ and $E_{\rm to}$ with $n$ is qualitatively similar to the example discussed in Remark \ref{rem.energyduality}, even though the columns of $F$ are generally non-orthonormal.
This relationship is not guaranteed to hold when actuators/sensors are placed optimally (rather than randomly) in the network. An example of the latter is given in Fig. \ref{fig.complexnets}c for chain networks of increasing size and specific choices of actuator, sensor, and target nodes. In this case, the energy $E_{\rm to}$ still decreases on average as $n$ increases whereas $E_{\rm tc}$ remains constant for all $n$.
This demonstrates that the cost of target control and estimation methods can be further optimized through the placement of actuators and sensors in the network (as previously explored for target control \cite{Casadei2020,li2020structural}). 
%

\section{Duality and Target Control}
\label{sec.separation}

Full-state controllability is a necessary and sufficient condition 
for the design of \textit{static feedback} control systems with arbitrary poles. 
On the other hand, its generalized counterpart, output controllability, has only been shown to be a necessary and sufficient condition for the control of the target vector $\bm z$ 
via the open-loop control law \eqref{eq.statedependentcontrol} (based on the Gramian) \cite{Lazar2020,Casadei2020}. Although optimal in terms of control energy, such control law is prone to parameter uncertainties and disturbances, and there is no known relation between output controllability and the design of a control system with {static feedback} $\bm u = -K\bm x$. 

In this section, we apply the duality principle to obtain a necessary and sufficient condition for \textit{target control via static feedback}.
Specifically, we show that, given a system $(A,B;F)$, there exists a static feedback matrix $K\in\R^{p\times n}$ and an input $\bm u = -K\bm x$ such that the closed-loop system 
\begin{equation}
    \dot{\bm x} = (A-BK)\bm x    
\label{eq.closedloop}
\end{equation}
can be controlled with respect to the target subspace \eqref{eq.target} if and only if the \textit{dual} system $(B^\transp,A^\transp;F)$ is functionally~observable.
In the classical pole placement problem, full-state controllability ensures that all the closed-loop poles (eigenvalues of $A-BK$) can be arbitrarily placed via static feedback. 
Here, instead, we solve a \textit{partial} pole placement problem in which the control goal is to design a static feedback that arbitrarily places only the subset of closed-loop poles $\Sigma$ (defined below) that directly influence the time response characteristics of $\bm z(t)$. 

Before stating the theorem, we define the eigenpair $(\lambda_i,\bm v_{ij})$ corresponding to an eigenvalue $\lambda_i$ of $A$ and its left eigenvector $\bm v_{ij}\in\C^n$, for $j=1,\ldots,k_i$, where $k_i$ is the geometric multiplicity of $\lambda_i$. Let $V_{i}\in\C^{n\times k_i}$ be an orthogonal basis of the eigenspace $\{\bm v\in\C^n \,| \, (A^\transp - \lambda_i I_n)\bm v \}$, $Q = F^\transp(FF^\transp)^{-1}F$ be the orthogonal projection onto $\operatorname{row}(F)$, and $|\mathcal J|$ be the number of elements in a set $\mathcal J$. 

\begin{thm}[Target control via static feedback] \label{thm.stroutctrl}
Consider a system $(A,B;F)$ with static feedback $\bm u = -K\bm x$. There exists a feedback matrix $K$ such that every eigenpair in
\begin{equation}
\begin{aligned}
    \Sigma = \bigl\{(\lambda_i,\bm v_{ij}) \, | \,& Q\bm v_{ij}\neq 0, \, j\in \mathcal J_i \subseteq \{1,\ldots, k_i\}, \, \\  &|\mathcal J_i| = \rank(QV_i) \bigr\}
\end{aligned}
\end{equation}
\noindent
can be arbitrarily placed in the closed-loop system \eqref{eq.closedloop} if and only if the dual system $(B^\transp,A^\transp;F)$ is functionally observable.
\end{thm}

\begin{proof}
Let $\mathcal C$ be the controllability matrix of $(A,B;F)$ and $\mathcal O = \mathcal C^\transp$ be observability matrix of $(B^\transp,A^\transp;F)$, and assume that $\operatorname{rank}(\mathcal C) = m < n$. 
We apply the following canonical decomposition \cite{Chi-TsongChen1999} to system $(A,B;F)$. Let $P^\transp = [p_1 \,\, \ldots \,\, p_n] \in\R^{n\times n}$ be a unitary matrix such that the first $m$ columns $\{p_1,\ldots, p_m\}$ lie in the column space of $\mathcal C$ and $\{p_{m+1},\ldots,p_n\}$ are arbitrarily chosen such that $P$ is nonsingular. \rev{Thus, $\bar{\bm x} := [\bm x_c^\transp \,\, \bm x_u^\transp]^\transp = P\bm x$ decomposes $\bm x$ into the controllable variables $\bm x_c\in\R^{m}$ and uncontrollable variables $\bm x_u\in\R^{n-m}$.} Applying this transformation to \eqref{eq.dynsys}--\eqref{eq.target} yields:
\begin{equation}
\begin{aligned}
     \begin{bmatrix}
         \dot{\bm x}_c \\ \dot{\bm x}_u
     \end{bmatrix}
     &=
     \rev{\underbrace{
     \begin{bmatrix}
         A_c & A_{12} \\ 0 & A_{u}
     \end{bmatrix}}_{\bar A}}
     \begin{bmatrix}
         {\bm x}_c \\ {\bm x}_u
     \end{bmatrix}
     +
     \rev{\underbrace{
     \begin{bmatrix}
         B_c \\ 0
     \end{bmatrix}}_{\bar B}}
     \bm u,
     \\
     \bm y &=
     \rev{\underbrace{
     \begin{bmatrix}
         C_c & C_u
     \end{bmatrix}}_{\bar C}}
     \begin{bmatrix}
         {\bm x}_c \\ {\bm x}_u
     \end{bmatrix},
     \,\,\, \text{and} \,\,\,
     \bm z =
     \rev{\underbrace{
     \begin{bmatrix}
         F_c & F_u
     \end{bmatrix}}_{\bar F}}
     \begin{bmatrix}
         {\bm x}_c \\ {\bm x}_u
     \end{bmatrix},
\end{aligned}
\label{eq.decomposedsys}
\end{equation}
\noindent
where $\bar A= PAP^\transp$, $\bar B = PB$, $\bar C = CP^\transp$,
and $\bar F = FP^\transp$.
Given that $[F_c \,\, F_u] = FP^{\transp}$, $F_u = 0$ if and only if $\operatorname{row}(F)\subseteq \operatorname{row}([p_1 \,\, \ldots \,\, p_m]^\transp) = \operatorname{row}(\mathcal C^\transp)$. Since $\mathcal O = \mathcal C^\transp$, this condition is equivalent to $\operatorname{row}(F)\subseteq \operatorname{row}(\mathcal O)$, which holds if and only if condition \eqref{eq.functobsv} is satisfied for a triple $(B^\transp,A^\transp;F)$. Thus, $F_u = 0$ if and only if $(B^\transp,A^\transp;F)$ is functionally observable.

Before proceeding with the rest of the proof, we first recall the subsystem $(A_c,B_c)$ is controllable and thus all eigenvalues (and eigenpairs) of $(A_c-B_c K_c)$ can be arbitrarily placed via static feedback $\bm u = -K\bm x$, where $KP^\transp = [K_c \,\, K_u]$ \cite[Th. 10]{Chi-TsongChen1999}.
%
For each eigenpair $(\lambda_i,\bm v_{ij})$ of $A$, there is a corresponding eigenpair $(\lambda_i,\bar{\bm v}_{ij})$ of $PAP^\transp$ sharing the same eigenvalue $\lambda_i$, where $\bar{\bm v}_{ij} = [(\bm v^c_{ij})^\transp \,\, (\bm v^u_{ij})^\transp]^\transp = P {\bm v}_{ij}$.
Under a transformation $P$, all elements of $\Sigma$ have a one-to-one correspondence to the elements of the subset of eigenpairs of $PAP^\transp$, defined as
%
    $\bar\Sigma = \bigl\{(\lambda_i,\bar{\bm v}_{ij}) \, | \, QP^\transp\bm \bar{\bm v}_{ij}\neq 0, \, j\in \mathcal J_i \subseteq \{1,\ldots, k_i\}, \, |\mathcal J_i| = \rank(QV_i) \bigr\}$.
%
Note that, due to the structure of the decomposed system \eqref{eq.decomposedsys}, we can construct a set of eigenpairs $(\lambda_i,\bar{\bm v}_{ij})\in\bar\Sigma$ such that all eigenvectors $\bar{\bm v}_{ij}$ satisfy either $\bar{\bm v}^c_{ij}=0$ or $\bar{\bm v}^u_{ij}=0$.

Now we show that every eigenpair $(\lambda_i,\bar{\bm v}_{ij})\in\bar\Sigma$ can be arbitrarily placed if and only if $F_u = 0$. 
The proof then follows from the fact that every $(\lambda_i,\bar{\bm v}_{ij})\in\bar\Sigma$ can be arbitrarily placed by some feedback matrix $KP^\transp$, which implies that every $(\lambda_i,{\bm v}_{ij})\in\Sigma$ can also be abitrarily placed by $K$.

\textit{Sufficiency.} Assuming $F_u = 0$, we show that any eigenpair $(\lambda_i,\bar{\bm v}_{ij}) \in \bar\Sigma$ can be arbitrarily placed when $\lambda_i$ has multiplicity $k_i=1$ and $k_i>1$. For non-repeated eigenvalues $\lambda_i$ ($k_i=1$), by definition $(\lambda_i,\bar{\bm v}_{i1})\in\bar\Sigma$ if $QP^\transp\bar{\bm v}_{i1} = F^\transp(FF^\transp)^{-1}
        [F_c \,\, F_u]
        [(\bm v^c_{i1})^\transp \,\, (\bm v^u_{i1})^\transp]^\transp
        \neq 0$.
%
%
\noindent
For any eigenpair $(\lambda_i,\bar{\bm v}_{i1})$ that cannot be arbitrarily placed, we have that $\bar{\bm v}_{i1}^c = 0$ and, since $F_u = 0$, it follows that $QP^\transp\bar{\bm v}_{i1}= 0$. Thus, any eigenpair $(\lambda_i,\bar{\bm v}_{i1})$ that satisfies $QP^\transp\bar{\bm v}_{i1}\neq 0$, and by definition belongs to $\bar\Sigma$, can be arbitrarily placed.

For each eigenvalue $\lambda_i$ with multiplicity $k_i>1$, since all eigenvectors $\bar{\bm v}_{ij}$ are linearly independent, the projected vectors $Q\bar{\bm v}_{ij}$ are also linearly independent due to the orthogonality of $Q$. Therefore, all $|\mathcal J_i|$ eigenvectors $\bar{\bm v}_{ij}$ in $\bar\Sigma$ are linearly independent under the projection $QP^\transp$. Let $\operatorname{diag}(I_m,0_{n-m})\bar{\bm v}_{ij}$ be a projection of the eigenvector $\bar{\bm v}_{ij}$ onto the controllable subspace. Since $F_u = 0$, then $\operatorname{row}(F) = 
\operatorname{row}([F_c \,\, 0_{r\times n-m}])\subseteq\operatorname{row}(\operatorname{diag}(I_m,0_{n-m}))$. Consequently, for each eigenvalue $\lambda_i$, all eigenvectors $\bar{\bm v}_{ij}$ in $\bar\Sigma$ are also linearly independent under the projection $\operatorname{diag}(I_m,0_{n-m})$. This implies that every $(\lambda_i,\bar{\bm v}_{ij})\in\bar\Sigma$ can be arbitrarily placed.


\textit{Necessity.} By assuming $F_u\neq 0$, we show that there exists at least one $(\lambda_i,\bar{\bm v}_{ij})\in\bar\Sigma$ that cannot be arbitrarily placed. Recall that every eigenpair $(\lambda_i, \bar{\bm v}_{ij})$ that cannot be arbitrarily placed satisfies ${\bm v}_{ij}^c = 0$, yielding $QP^\transp\bm s = F^\transp(FF^\transp)^{-1}F_u{\bm v}_{ij}^u$. Note that $F_u\in\R^{r\times n-m}$ and hence $\rank(F_u)\leq n-m$. Given that the eigenpairs $(\lambda, \bar{\bm v}_{ij})$ that cannot be arbitrarily placed span an $n-m$ dimensional space, there exists at least one eigenpair $(\lambda,\bar{\bm v}_{ij})$ that cannot be arbitrarily placed and satisfies $F^\transp(FF^\transp)^{-1} F_u{\bm v}_{ij}^u \neq 0$, implying that $(\lambda,\bar{\bm v}_{ij})\in\bar\Sigma$.
\end{proof}

\begin{rem}
    Theorem~\ref{thm.stroutctrl} does not depend on a specific selection of eigenvectors of $A$ and hence holds for any feasible set $\Sigma$. If $A$ is diagonalizable, then $\bm z = F\sum_i^n \bm v_{i1}^\transp \bm x(0)\bm v'_{i1} e^{\lambda_i t}$, where $\bm v_{i1}$ and $\bm v'_{i1}$ are respectively the left and right eigenvectors associated with $\lambda_i$. The dynamics of $\bm z(t)$ depend only on eigenvalues $\lambda_i$ that satisfy $F\bm v'_{i1}\neq 0$. From Theorem \ref{thm.stroutctrl}, any $(\lambda_i,\bm v_{i1})$ satisfying $F\bm v'_{i1} \neq 0$ belongs to $\Sigma$ and can be placed.
\end{rem}

Following the proof of Theorem~\ref{thm.stroutctrl}, if $(B^\transp,A^\transp;F)$ is functionally observable, then $\bm z = F\bm x = F_c\bm x_c$ is only a function of the controllable variables. Thus, Theorem~\ref{thm.stroutctrl} establishes a necessary and sufficient condition for the stabilization problem of $\bm z(t)$: for any 
$\alpha\in\C$, there exists a matrix $K$ such that the solution of $\dot{\bm x} = (A - BK + \alpha I_n)\bm x$ satisfies $\lim_{t \rightarrow \infty} (F\bm x(t) - \bm z^*) = 0$ for any initial condition $\bm x(0)$ and some state $\bm z^*\in\R^n$. Note that this formulation does not account for the placement of unstable poles outside of $\Sigma$, which may lead to instability in other state variables (cf. \cite[Sec. 5.6]{wonham1979book}).

Fig.~\ref{fig.outctrbviafeed} summarizes the established relation between output controllability, functional observability, and the design of static feedback systems. A condition for the existence of a solution to the partial pole placement problem is also proposed in \cite[Th.~4.4]{wonham1979book}. This theorem states that every eigenpair in $\Sigma$ can be arbitrarily placed in if and only if $\mathcal X_{\Sigma} \subseteq \langle A | \Im(B) \rangle + \mathcal V'$, where $\mathcal X_{\Sigma}=\{\bm v_{ij}\in\R^n : (\lambda_i,\bm v_{ij})\in\Sigma\}$ is the modal subspace of $A$ that influences the dynamics of $\bm z(t)$ and $\mathcal V'$ is the supremal element among all the $(A,B)$-invariant subspaces contained in $\operatorname{Ker}(F)$. Despite the strong geometric interpretation of this theorem, testing the condition above for $\mathcal X_{\Sigma}$ requires determining $\mathcal V'$, which can be computationally demanding for large-scale systems. In contrast, Theorem \ref{thm.stroutctrl} provides an equivalent, but much simpler, existence condition that is solely based on the rank test \eqref{eq.functobsv} for $(B^\transp,A^\transp;F)$.

\begin{figure}[t!]
\begin{center}
\includegraphics[width=0.95\columnwidth]{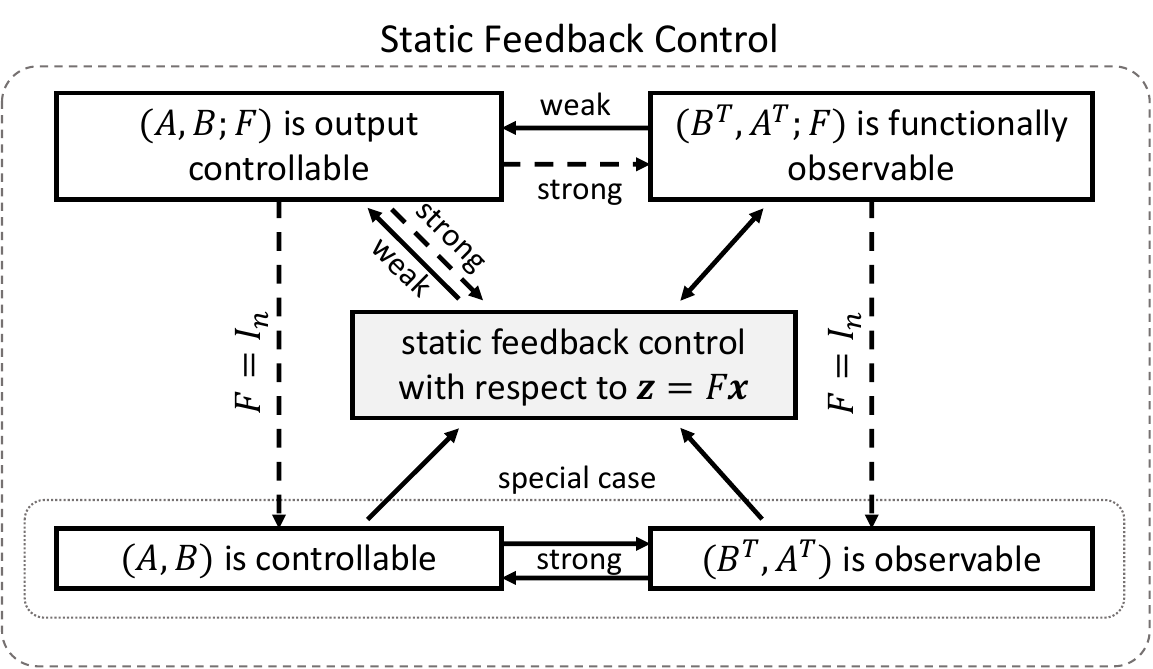} 
\caption{\label{fig.outctrbviafeed} Relation between output controllability, functional observability, and static feedback system design for target control (with arbitrary pole placement). Solid arrows indicate one property implies another, while dashed arrows indicate conditional implications.}
\end{center}
\vspace{-0.8cm}
\end{figure}

\vspace{-0.2cm}
\subsection{Separation principle}

Building on the previous result, we investigate the problem of \textit{target control via (observer-based) output feedback}. The next result establishes a separation principle, stating that the design of target feedback controllers and functional observers are mutually independent in closed-loop systems if the systems $(B^\transp,A^\transp;F)$ and $(C,A;-K)$ are functionally observable. The functional observability of $(B^\transp,A^\transp;F)$ implies the existence of some $K$ such that the poles in $\Sigma$ can be arbitrarily placed in closed loop (following Theorem \ref{thm.stroutctrl}), whereas the functional observability of $(C,A;-K)$ guarantees that the designed input $\bm u = -K\bm x$ can be directly estimated from the output $\bm y$ using a functional observer (following Refs. \cite{Darouach2000,Fernando2010}).

\begin{cor} [Separation principle]     \label{theor.separationprinciple}
Consider~a~system \eqref{eq.dynsys}--\eqref{eq.output} coupled with the functional observer-based feedback controller
\begin{align}
    \dot{\bm w} = N\bm w + J\bm y, \label{eq.estimationcontroldyn} \\
    {\bm u} = D\bm w + E\bm y, \label{eq.estimationcontrolout}
\end{align}

\noindent
where $\bm w\in\R^{n_0}$ is the internal state of the functional observer and $\bm u$ is both the output of the observer and the controller signal. Let the target vector sought to be controlled be defined by \eqref{eq.target}. The closed-loop system has a characteristic polynomial
\begin{equation}
    p(\lambda) = \operatorname{det}(\lambda I_n - A + BK)\operatorname{det}(\lambda I_{n_0} - N) .
\label{eq.charactpoly}
\end{equation}
\noindent
There exist matrices $(K,N,J,D,E)$ such that the roots $p(\lambda_i)=0$ can~be arbitrarily placed for every $(\lambda_i,\bm v_{ij})\in\Sigma$ and every eigenvalue $\lambda_i$ of matrix $N$ if the triples $(B^\transp,A^\transp;F)$ and $(C,A;-K)$ are both functionally observable.
\end{cor}

\begin{proof}
Substituting \eqref{eq.estimationcontroldyn} and \eqref{eq.estimationcontrolout} into \eqref{eq.dynsys} leads to 
\begin{equation}
    \begin{bmatrix}
        \dot{\bm x} \\ \dot{\bm w}
    \end{bmatrix}
    =
    \begin{bmatrix}
        A + BEC & BD \\
        JC & N
    \end{bmatrix}
    \begin{bmatrix}
        {\bm x} \\ {\bm w}
    \end{bmatrix}
    .
    \label{eq.feedbackestimatesys}
\end{equation}
Let $-K\bm x$ be the functional sought to be estimated by the functional observer \eqref{eq.estimationcontroldyn}--\eqref{eq.estimationcontrolout}. Since $(C,A;-K)$ is functionally observable by assumption, there exist matrices $(N,J,D,E)$ such that the following conditions are satisfied \cite[Th.~3]{Fernando2010}: i) $N$ is Hurwitz, ii) $NT + JC - TA = 0$, and iii) $- DT - EC = K$. From \cite[Th. 1]{Darouach2000}, these conditions further imply that the functional observer is stable (i.e., $\lim_{t\rightarrow\infty}({\bm u} (t) +K\bm x (t)) = 0$) and, equivalently, $\lim_{t\rightarrow\infty}(\bm w (t) - T\bm x(t)) = 0$ for some matrix $T$. Applying equalities (ii)-(iii) and $\bm e = \bm w - T\bm x$ to \eqref{eq.feedbackestimatesys} yield
\begin{equation}
    \begin{bmatrix}
        \dot{\bm x} \\ \dot{\bm e}
    \end{bmatrix}
    =
    \begin{bmatrix}
        A - BK & BD \\ 0 & N
    \end{bmatrix}
    \begin{bmatrix}
        {\bm x} \\ {\bm e}
    \end{bmatrix}.
    \label{eq.separationprinciple}
\end{equation}
\noindent
The eigenvalues of the matrix above depend exclusively on $(A-BK)$ and $N$, leading to Eq.~\eqref{eq.charactpoly}.
Let system \eqref{eq.separationprinciple} be decomposed as \eqref{eq.decomposedsys} via $\bar{\bm x} = P\bm x$, yielding
$\operatorname{det}(\lambda I_n - A+BK) = \operatorname{det}(\lambda I_m - A_c + B_c K_c) \operatorname{det}(\lambda I_{n-m} - A_u)$. Since $(A_c,B_c)$ is controllable, there exists some $KP^\transp = [K_c \,\, K_u]$ such that all roots of $\operatorname{det}(\lambda I_m - A_c + B_c K_c)$ can be arbitrarily placed. From Theorem~\ref{thm.stroutctrl}, if $(B^\transp,A^\transp;F)$ is functionally observable, any $(\lambda_i,\bm v_{ij})\in\Sigma$ corresponds to an eigenvalue $\lambda_i$ of $A_c$. 
\end{proof}

\arthur{Corollary~\ref{theor.separationprinciple}} includes, as a special case, the classic separation principle. In this case, $\Sigma$ contains all eigenpairs of $A$ and the functional observer reduces to a full-state observer.

\vspace{-0.2cm}
\subsection{Example}
\label{sec.separationex}

\begin{figure}[b]
\vspace{-0.4cm}
    \centering
    \includegraphics[width=0.75\columnwidth]{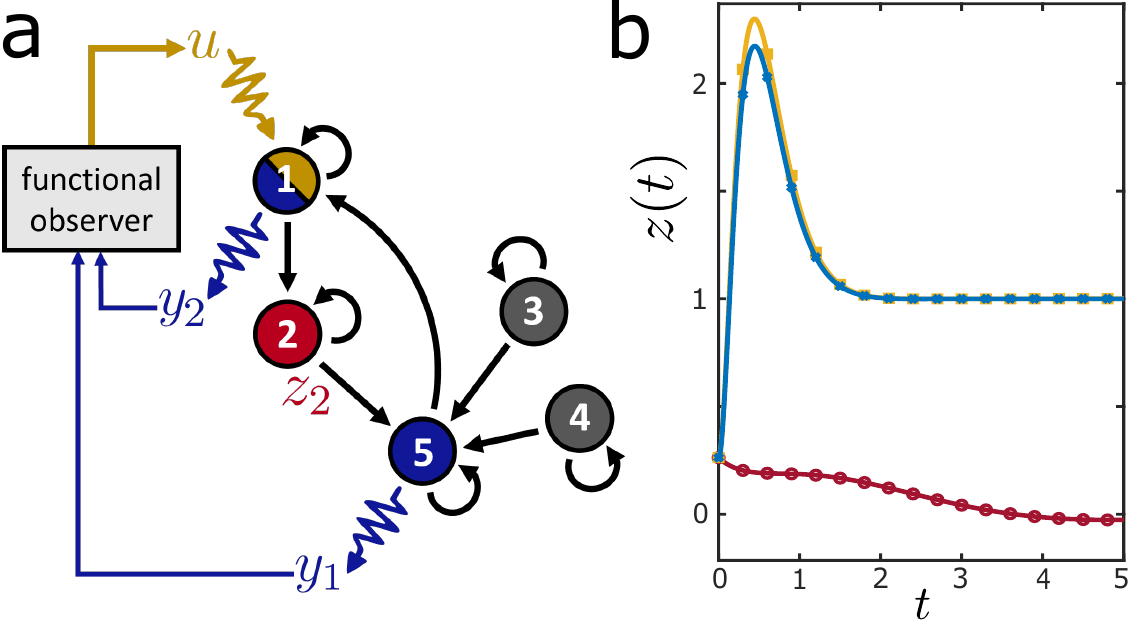}
    \caption{
    \label{fig.separation} Target control via output feedback. (a) Graph representation of the closed-loop system for target control via output feedback. 
    Note that the input signal is the observer's output. 
    (b) State evolution of $\bm z(t) = x_2(t)$ in the open-loop system (red), closed-loop system via full-state feedback (blue), and closed-loop system via functional observer-based output feedback (yellow). 
    }
\end{figure}

Consider system \eqref{eq.dynsys}--\eqref{eq.target}, where $A$ and $F$ are given by \eqref{eq.pairACexample} ($a_{22} = -1$, $a_{33}=a_{44}=-3$), $B = \bm e_1$, $C = [\bm e_1,\bm e_5]^\transp$, and $\bm e_i$ is a basis vector with 1 in the $i$th entry and 0 elsewhere. Fig.~\ref{fig.separation}a shows a graph of the closed-loop system.  
%
%
%
The system is uncontrollable, but the dual system $(B^\transp,A^\transp;F)$ is functionally observable. 
%
%
Thus, following Theorem~\ref{thm.stroutctrl}, the eigenvalues $\lambda = \{-2, -0.5\pm i0.866\}$ of $A$, which comprise set $\Sigma$, can be arbitrarily placed via feedback. 
Note that the eigenvalues $\lambda = -3$ do not belong to $\Sigma$ and cannot be placed. 

As an example of closed-loop design, consider the following specifications for the response of $\bm z(t)$: place the poles in $\Sigma$ at $\{-4, -5, -6\}$ and stabilize it at setpoint $\bm z^* = 1$. 
This can be achieved using the full-state feedback $\bm u = \bm r - K\bm x$, where $K = [12, -47, 0, 0, 59]$ and $\bm r = -120$.  
%
%
%
%
Since $(C,A;-K)$ is functionally observable, we can also solve this problem via output feedback by designing the functional observer 
as per \cite[Sec. 3.6.1]{Trinh2012}. This yields $N = -1$, $J = [0, -47]$, $D = 1$, $E = [-59, -12]$, and $T = [0, 47, 0, 0, 0]$. Following the separation principle, the design of the feedback matrix and functional observer were carried out independently. Fig.~\ref{fig.separation}b shows simulations for the open-loop and closed-loop systems.




\vspace{-0.25cm}
\section{Conclusion}
\label{sec.conclusion}


%

The duality established here is expected to allow techniques developed for functional observability to be employed in the study of output controllability, and vice versa. 
As shown in a recent application to structured systems \cite{Montanari2023target}, the strong duality principle enables algorithms developed to optimally place actuators for target control in large-scale networks \cite{Gao2014,Waarde2017,Czeizler2018} to be used to optimally place sensors for target estimation \cite{Montanari2022}.
Based on similar arguments, we expect that methods developed for the design of functional observers can also be used to design target controllers, and vice versa. In particular, the open problem of designing target controllers via \textit{dynamic} feedback could then be approached by mapping the solution from a dual problem of functional observer design, which has well-established scalable solutions \cite{Darouach2000,Fernando2010,Trinh2012,Habibi2022,Montanari2022}. 
For example, in the control of cluster synchronized states, one may leverage design methods of functional observers originally developed for the estimation of average cluster states \cite{niazi2020average,niazi2023clustering}. 
Finally, we anticipate that our results may be extended to broader classes of systems, including stochastic \cite{Georgiou2013}, nonlinear \cite{Montanari2022nonlinear}, and differential-algebraic \cite{Campbell1991} systems.



\vspace{-0.3cm}

\end{document}